\documentclass[12pt]{article}

\usepackage{amsmath}
\usepackage{amssymb}
\usepackage{amsthm}
\usepackage{graphicx}
\usepackage{amscd}
\usepackage{epic, eepic}
\usepackage{url}
\usepackage{color}

\newtheorem{theorem}{\bf Theorem}[section]
\newtheorem{lemma}[theorem]{\bf Lemma}

\newtheorem{problem}[theorem]{\bf Problem}
\newtheorem{prop}[theorem]{\bf Proposition}

\newtheorem{nota}[theorem]{\bf Notation}

\newcommand{\cP}{\mathcal{P}}
\newcommand{\cL}{\mathcal{L}}

\newcommand{\PG}{\mathrm{PG}}

\newcommand{\cut}[1]{}

 \makeatletter
 \@addtoreset{equation}{section}
 \makeatother

 \makeatletter
 \@addtoreset{table}{section}
 \makeatother

\title{Saturating sets in projective planes and hypergraph covers}

\author{
  {\bf Zolt\'an L\'or\'ant Nagy}\thanks{The author is supported by the Hungarian Research Grant (OTKA) No. K 120154 and by the J\'anos Bolyai Research Scholarship of the Hungarian Academy of Sciences}\\ \\ 
{\small MTA--ELTE Geometric and Algebraic Combinatorics Research Group}\\
{\small H--1117 Budapest, P\'azm\'any P.\ s\'et\'any 1/C, Hungary}\\
{\small \tt{  nagyzoli@cs.elte.hu}}}
\date{}

\begin{document}
\maketitle

\begin{abstract} 
Let $\Pi_q$ be an arbitrary finite projective plane of order $q$. A subset $S$ of its points is called saturating if any point outside $S$ is collinear with a pair of points from $S$.
Applying probabilistic tools  we improve the upper bound on the smallest possible size of the saturating set to $\lceil\sqrt{3q\ln{q}}\rceil+ \lceil(\sqrt{q}+1)/2\rceil$.
The same result is presented using an algorithmic approach as well, which points out the connection with the transversal number of uniform multiple intersecting hypergraphs.

\end{abstract}

\section{Introduction}
Let $\Pi_q$ be an arbitrary finite projective plane of order $q$, and $\cP, \cL$ denote the point and line set of the plane, respectively. A subset $S$ of its points is called a saturating set if any point outside $S$ is collinear with two points in $S$. In other words, the secants of the point set cover the whole plane.\\
The size of saturating (or sometimes also called \textit{dense, saturated, determined}) sets has been widely investigated, see \cite{survey} for a recent survey. The importance of such sets relies on connections to covering codes \cite{Davy}, algebraic curves over finite fields \cite{Giu}, sumset theory \cite{Grin} and complete arcs \cite{survey}. In fact, there is a one-to-one correspondence between saturating sets in $\PG(2,q)$ of cardinality $|S|$ and linear $q$-ary covering codes of codimension $3$ and covering radius $2$, of length $|S|$. Here we only refer the reader to    \cite{Cohen, Davy} for the details.

\begin{nota} For any points $P$ and $Q$, the set $\left\langle P, Q\right\rangle$ denotes the points of the line determined by $P$ and $Q$. In general, if $S$ and $T$ are point sets, then the set $\left\langle S,T\right\rangle$ denotes the points of the lines determined by a point from $S$ and a point from $T$. For convenience, if one of the sets consists of a single point, we use $P$ instead of $\{P\}$. \\ The set $\left\langle S,S \right\rangle$ denotes the points of the lines determined by a distinct point pair from $S$.
\end{nota}

Clearly, the notation rewrites the concept of saturating sets to the following property on a set $S$: $\left\langle S, S \right\rangle=\cP(\Pi_q)$.

An \textit{arc} in the projective plane is a set of points such that no three points of the set are on a line,
and  arcs which can not be extended to a larger arc are called \textit{complete arcs}. Observe that complete arcs are special separating sets. Due to the pioneer work of  Ughi, Sz\H{o}nyi, Boros, Giulietti, Davydov, Marcugini, Pambianco, Kim and Vu \cite{Szonyi, DMP, survey, kimvu, Ughi} several bounds are known for particular finite projective planes concerning maximal arcs and saturating sets.

It is well known (and mentioned first in \cite{Ughi} concerning saturating sets) that the following proposition must hold:

\begin{prop}  $ |S|> \sqrt{2q}+1$ for any saturating set $S$ in $\Pi_q$.
\end{prop}

 This estimate is also known as the Lunelli-Sce bound for complete arcs.

If $q$ is a square and the plane is Desarguesian, then the existence of saturating sets with the same order of magnitude up to a constant factor is known due to Boros, Sz\H onyi and Tichler \cite{Szonyi}.

\begin{prop}\cite{Szonyi, Ughi}     The union of three non-concurrent Baer sublines  in a Baer subplane from a Desarguesian plane of square order is a
saturating set of size $3\sqrt{q}$.
\end{prop}
In fact,  this can be even generalized since any $2$-blocking set of a Baer-subplane,  which meets every line in at least $2$ points, provides a saturating set of the ground plane. Thus if $q$ is a $4$th power, this bound can be even improved to roughly $2\sqrt{q}$ due to Davydov et al. \cite{DG} and  Kiss et al. \cite{kiss}. Similar order of magnitude can be achieved if  $q$ is a $6$th power \cite{DG}. 

\begin{prop}\cite{DG, kiss}     The union of the point set of two disjoint  Baer subplane  of a Baer subplane from a Desarguesian plane is a
saturating set, of size  $2\sqrt{q} + 2\sqrt[4]{q}+2$. 
\end{prop}
(The existence of Baer-subplanes in  Baer subplanes of $\PG(2,q)$ tacitly implies that $q$ is a $4$th power.)

However, in the general case when the plane is not necessarily Desarguesian or the order is arbitrary, we only have much weaker results. Following the footprints of  
 Boros, Sz\H onyi and Tichler \cite{Szonyi}, Bartoli, Davydov, Giulietti, Marcugini and Pambianco \cite{Bartoli} obtained an estimate on the minimal size of a saturating  set in $\Pi_q$.
\begin{prop}[Bartoli et al.] \cite{Bartoli}   
$\min |S| \leq \sqrt{4(q+1)\ln{(q+1)}}+2$, if $S$ is a saturating set in  $\Pi_q$.
\end{prop}

Our main theorem improves the constant term to $\sqrt{3}$.

\begin{theorem}
\label{main}   
$\min |S| \leq (\sqrt{3}+o(1))\sqrt{q\ln{q}}$,  if $S$ is a saturating set in  $\Pi_q$.
\end{theorem}

In fact, we prove the exact result $ \min |S| \leq \lceil\sqrt{3q\ln{q}}\rceil+ \lceil(\sqrt{q}+1)/2\rceil$.

Computer searches suggest that for Galois planes, this bound is still not sharp and the correct order of magnitude is probably $O(\sqrt{q})$, see \cite{DMP}. However, we will see in Section 4 that our estimate might be sharp in general, if we cannot built on any algebraic structure of the plane.
We focus mainly on the order of magnitude of $\min |S|$, partly because it has been determined for $q$ small in $PG(2,q)$, see \cite{B1, B2} and the references therein. 
In particular, Bartoli et al. showed  an  upper bound slightly above $(\sqrt{3})\sqrt{q\ln{q}}$ for the smallest size of complete arcs  in $PG(2,q)$ under a certain probabilistic conjecture, and their computer search results suggest that this is indeed the right order of magnitude \cite{B1, B2}. These results also support the assumption that Theorem \ref{main} might be sharp for general planes. 
%   q small: nincs itt látnivaló.  speaks in favor of this bound: 3,4,5.

The paper is built up as follows. In Section 2, we present the first proof of Theorem \ref{main}, applying a refined version of the first moment method, where an almost suitable structure is proved via random argument, which can be fixed. This idea appeared first in Erd\H{o}s' work, who proved this way the existence of dense, complete bipartite graph free and dense, even cycle free graphs, see e.g. in \cite{Ball} (Chapter 6). In Section 3, we proceed by showing that an advanced greedy-type algorithm   also provides a saturating set of this size. Finally, in Section 4 we analyze the above approaches and point out their connection to hypergraph cover problems, which suggests that probably in general  (non-desarguesian) projective planes it would be hard to improve the order of magnitude. Note that similar connection between geometric problems and transversals in hypergraph appeared before several times \cite{Alon1}.   We finish  with a number of open problems.

\section{Probabilistic argument  - First proof  }

\begin{proof}[Proof of Theorem \ref{main}]

The key idea is  the following: we put in every point to our future saturating set $S^*$ with a given probability $p:=p(q)$ determined later on, independently from the other points. Then we complete it to obtain a saturating set via Lemma \ref {felezes}. 

\begin{lemma}\label{felezes}
Consider a set of points $S^*$ and the corresponding set $R$ containing those points which are not determined by the lines of $S^*$. One can add at most $\lceil|R|/2\rceil$ points to $S^*$ so that the resulting set is a saturating set.
\end{lemma}

\begin{proof}
Pair up the points of $R$ and for each pair $\{x, x'\}$, choose the intersection of $\left\langle x, s\right\rangle$ and $\left\langle x,' s'\right\rangle$ for two different points $s, s'$ in $S^*$. Clearly, the addition of the intersection makes the points $x, x'$ determined.
\end{proof}

Clearly if $X$ is random variable counting the number of points in $S^*$, we have $\mathbb{E}(X)=p(q^2+q+1)$. 
Then if $Y$ is random variable counting the number of points not in $\left\langle S^*, S^* \right\rangle$ we get

\begin{equation}\label{eq:Y}
\begin{aligned}
\mathbb{E}(Y)&=(q^2+q+1)\mathbb{P}(\mbox{a given point is not determined  })\\&= (q^2+q+1)(1-p)^{(q^2+q+1)}\left(\frac{p}{1-p}+ \left(1+q\frac{p}{1-p}\right)^{q+1}  \right).
\end{aligned}
\end{equation}

Indeed, in order to obtain that a given point is not determined, either a point is chosen, thus no further points may appear; or  it is not chosen and on every line on the point at most one point is chosen. The expression we get this way simplifies to (\ref{eq:Y}).

It is easy to see that $(q^2+q+1)(1-p)^{(q^2+q+1)}\left(\frac{p}{1-p} \right)$ does not contribute to the main term of $\mathbb{E}(Y) $, so we omit it.

We want to choose the value of $p:=p(q)$ in such a way that the sum $\mathbb{E}(X)+\lceil\frac{1}{2}\mathbb{E}(Y)\rceil$ is minimized, which would provide a suitable saturating set via the first moment method, in view of Lemma \ref{felezes}.  This yields $p\approx \frac{1}{2}\exp{\left(\frac{2p}{p-2}(q^2+q+1)\right)}\exp{\left((q+1)\cdot\frac{2\frac{pq}{1-p}}{\frac{pq}{1-p}+2}  \right)}$ if we apply the approximation forms via Taylor's theorem

$$ (1+1/x)^{x+0.5}=e\cdot\left(1+ \frac{1}{12x^2}-O\left(\frac{1}{x^3}\right)\right)\ \ \ (x\rightarrow \infty) \Leftrightarrow 1+z \approx \exp{\left(\frac{2z}{z+2}\right)}(1+O(z^3)) \ \ \ (z\rightarrow 0),$$  \mbox{for}  $ z=\frac{pq}{1-p}$,  
as $\left(1+ \frac{z^2}{12}-O\left({z^3}\right)\right)^{\frac{2z}{z+2}}=(1+O(z^3))$ if $z\ll 1$, and

$$ (1-1/x)^{x-0.5}= e^{-1}\cdot\left(1-\frac{1}{12x^2}-O\left(\frac{1}{x^3}\right)\right)\ \ \ (x\rightarrow \infty) \Leftrightarrow 1-p \approx \exp{\left(\frac{2p}{p-2}\right)}(1-O(p^3)) \ \ \ (p\rightarrow 0).$$ 

From this, one can derive

 $$\frac{1}{2}\mathbb{E}(Y) =\frac{1}{2}(q^2+q+1)\exp{\left(\frac{2p}{p-2}(q^2+q+1)\right)}\cdot\Delta_1^{q^2+q+1}\exp{\left((q+1)\cdot\frac{2\frac{pq}{1-p}}{\frac{pq}{1-p}+2}  \right)}\cdot\Delta_2^{q+1},$$ 
where $\Delta_1= (1-O(p^3))$ and $\Delta_2=(1+O(z^3)) $ denote the error terms with $ z=\frac{pq}{1-p}$.

By simplifying the main term in $\mathbb{E}(Y)$, we get
\begin{equation}\label{eq:Y2}
\begin{aligned}
 \frac{1}{2}(q^2+q+1)\cdot\exp{\left( \frac{2p}{p-2}(q^2+q+1)\right)}\exp{\left((q+1)\cdot\frac{2\frac{pq}{1-p}}{\frac{pq}{1-p}+2}  \right)} &=\\
\frac{1}{2}(q^2+q+1)\cdot \exp{\left( \frac{2p}{p-2} + 2p(q^2+q)\cdot(\frac{1}{p-2}+\frac{1}{pq+2(1-p)}) \right)} &\approx\\
 \frac{1}{2}(q^2+q+1)\cdot \exp{\left( \frac{2p}{p-2} + 2pq(q+1)\cdot\frac{-p(q-1)}{4} \right)} 
  &\approx\\ \frac{1}{2}(q^2+q+1)\cdot  \exp{\left(-\frac{1}{2}q(q^2-1)p^2\right)},
\end{aligned}
\end{equation}
after we omit the smaller order terms during the approximation.

The  calculation above leads to the choice $p=\frac{\sqrt{3}\sqrt{q \ln{q}}}{q^2+q+1}$, which in fact implies that the main term in
$\mathbb{E}(Y)$ (\ref{eq:Y2}) equals to $O(\sqrt{q})$, and it is easy to check that the same holds for the error terms coming from $\Delta_1$ and $\Delta_2$. This in turn provides the existence of a saturating set of size  $$\mathbb{E}(X)+\lceil\frac{1}{2}\mathbb{E}(Y)\rceil=(1+o(1))\sqrt{3q\ln{q}}.$$
\end{proof}

\section{Algorithmic approach - Second proof}\label{algo}

Below we will present an algorithm to choose the point set of a saturating set $S$.  We start with an empty set $S_0$ in the beginning, and increase its cardinality by adding one point in each step.

\begin{nota}
In the $i$th step, we  denote the current set $S_i$ that will be completed to a saturating set, $D_i$  denotes the points of the plane outside $S_i$ which are determined by $S_i$ and $R_i$ denotes the set of points not in $S_i\cup D_i$. For a point set $H$, $\sigma{(H)}$ denotes the set of lines skew to $H$. 

The {\em benefit} of a point $b(P)$ (in step $i$) is the amount of points from $R_i$ which would become determined by the point set $S_i \cup \{P\}$, that is, $b(P)= |\left\langle P, S_i\right\rangle \cap R_i |$. 

%A multiplicity $m(\ell)$ of a line $\ell$ will measure its intersection with $D_i$: $m(\ell)=| D_i \cap \ell |$. The key role belongs to the tangent lines  $\tau(S_i)$ to $S_i$. $\tau(p)$ denotes the number of lines going through $p$ which are tangent to the set $S_i$.  
 \end{nota}

To obtain $D_{i+1}$ from $D_i$, we would like to add a point to $S_i$ which has the largest benefit. Consider all the lines skew to $S_i$, and choose one of them for which the intersection with $R_i$ is minimal. Adding up the benefits of the points of this line $\ell^*$, we get

\begin{equation}\label{eq:1}
 \sum_{P \in \ell^*} b(P) = |R_i \cap \ell^*|+i\cdot|R_i\setminus \ell^*|. 
\end{equation}

Indeed, the double counting counts any point $P\in R_i \cap \ell^*$ exactly once as these  points become determined only if we choose $P$ itself from $ \ell^*$. On the other hand, any point $Q$ of $R_i$ outside $\ell^*$ will be determined by adding a point from $ \ell^* \cap \left\langle Q, S_i\right\rangle$. This latter point set is of cardinality $i$ as $Q$ was not determined before.

\begin{lemma}\label{metszet} $$ \min_{\ell \in \sigma(S_i)} |R_i \cap \ell|\leq \frac{|R_i|}{q}.    $$
\end{lemma}

\begin{proof} There are at least $(q^2+q+1) - iq-1 = q(q+1-i)$ skew lines to any set of $i$ points. Every point of $R_i$ is appearing on exactly $(q+1-i)$ lines from $\sigma(S_i)$, hence 

$$ \sum_{\ell \in \sigma(S_i)}|R_i \cap \ell|= |R_i|(q+1-i),$$ thus the statement follows.
\end{proof}

\begin{prop}\label{fo} $$|R_{i+1}| \leq |R_i|\left(1-\frac{i}{q+2} \right) $$ if we add the point having the largest benefit to $S_i$.
\end{prop}

\begin{proof} 
We have $b(P)\geq \frac{1}{q+1}\left(|R_i \cap \ell^*|+i\cdot|R_i\setminus \ell^*| \right) $ 
for the point $P$ having the largest benefit on the skew line $\ell^*$ of minimal intersection with $R_i$, in view of Equation \ref{eq:1}.
Observe that $ |R_i \cap \ell^*|\leq \frac{|R_i|}{q}$, according to Lemma \ref{metszet}.
 By adding this point $P$ to $S_i$, we get $$|R_{i+1}|=|R_i|-b(P)\leq   |R_i| - \frac{ \frac{|R_i|}{q} + i(|R_i|-\frac{|R_i|}{q}) }{q+1}\leq |R_i|\left(1-\frac{i(q-1)}{q(q+1)} \right)  < |R_i|\left(1-\frac{i}{q+2} \right).$$
\end{proof}

\begin{lemma}\label{anal} $$\prod_{i=1}^{k:=\lceil\sqrt{3q\ln{q}}\rceil} \left(1-\frac{i}{q+2} \right) < {q}^{-3/2}.$$
\end{lemma}

\begin{proof} Denote $\prod_{i=1}^{k} \left(1-\frac{i}{q+2} \right) = \frac{(q+1)!}{(q+2)^k (q-k+1)!}$ by $A_q(k)$. We apply Stirling's approximation   $$\frac{n!}{t!}< \frac{\sqrt{2\pi n}\left(\frac{n}{e}\right)^n}{ \sqrt{2\pi t}\left(\frac{t}{e}\right)^t}$$ for $t<n$ with $t= q-k+1$ and $n=q+1$.
This implies 
$$A_q(k)< \frac{\sqrt{(q+1)}}{\sqrt{(q-k+1)}} \cdot \frac{(q+1)^{q+1}}{(q+2)^k\cdot (q-k+1)^{q-k+1}\cdot e^k}.$$

One again we use the Taylor-form for the approximation $(1+1/x)^{x+0.5}<   e\cdot\left(1+ \frac{1}{12x^2}\right)\ \  (x~>~1)$,

%\left(\frac{q+1}{q+2}\right)^k\frac{1}{e^k}        \left(1-\frac{1}{q+2}\right)^k\frac{1}{e^k} \approx       \left(\frac{q+1}{q+2}\right)^k                = \left( 1- \frac{(2+k)(q+2)-(k+1)}{(q+2)^2} \right)^{k/2}

 to obtain 

\begin{equation}\begin{aligned}\label{eq3} \left(1+\frac{k}{q-k+1}\right)^{q-k+1} =   \left(1+\frac{k}{q-k+1}\right)^{\left(\frac{q-k+1}{k}+1/2\right)k}  \left(1+\frac{k}{q-k+1}\right)^{-k/2} < \\
e^k\cdot  \left(1+  \frac{k^2}{12(q-k+1)^2}    \right)^k   \left(\frac{q-k+1}{q+1}\right)^{k/2}.
\end{aligned}
\end{equation}

Applying the approximation also for $(1-1/x)^{x-0.5}<~e^{-1}\cdot\left(1- \frac{1}{12x^2}\right)$ \ $(x\rightarrow \infty)$, we can simplify 

\begin{equation}\begin{aligned} \label{eq4}
 \left(\frac{q-k+1}{q+1}\right)^{k/2} \left(\frac{q+1}{q+2}\right)^k = \left( 1- \frac{(q+2)(k+2)+1}{(q+2)^2} \right)^{k/2}<  \left( 1- \frac{(k+2)}{(q+2)} \right)^{k/2}= \\
 \left( 1- \frac{(k+2)}{(q+2)} \right)^{ \left( \frac{q+2}{k+2}-0.5\right)\cdot k/2\cdot\left( \frac{q+2}{k+2}-0.5\right)^{-1}   }<
\left(  e^{-1}\cdot\left(1-\frac{(k+2)^2}{12(q+2)^2} \right)  \right) ^{\frac{k(k+2)}{2q-k+2}}. 
\end{aligned}
\end{equation}
%=\\e^{-\left((1+\epsilon)\frac{3}{2}\log{q}\right),

In total, if we take into consideration both (\ref{eq3}) and (\ref{eq4}), we get that 

$$A_q(k)< exp\left(-\left(\frac{k(k+2)}{2q-k+2}\right)\right)\Delta_3(k),$$

where $\Delta_3(k)$ is the product of error terms $ \frac{\sqrt{(q+1)}}{\sqrt{(q-k+1)}}$, $ \left(1+  \frac{k^2}{12(q-k+1)^2}    \right)^k$ and\\ $\left(1-\frac{(k+2)^2}{12(q+2)^2} \right)^{\frac{k(k+2)}{2q-k+2}}.$

If one plugs in $k= \lceil\sqrt{3q\ln{q}}\rceil$, careful calculations on the error term imply that $$A_q(k)<e^{-\left((1+\epsilon)\frac{3}{2}\ln{q}\right)},$$
 where $\epsilon= \epsilon(q)>0$.
\end{proof}

 By  applying Proposition \ref{fo} successively $k= \lceil\sqrt{3q\ln{q}}\rceil$ times it follows from Lemma \ref{anal} that we end up with at most $\sqrt{q}+1$ points remaining in $R_k$. But note that these points can be covered by adding a further $\lceil(\sqrt{q}+1)/2\rceil$ points to the saturating set, via Lemma \ref{felezes},
hence this algorithm provides a saturating set of size $ \lceil\sqrt{3q\ln{q}}\rceil+ \lceil(\sqrt{q}+1)/2\rceil$.

\section{Connections with hypergraph coverings, applications and open problems}\label{fin}

%If we take a closer look, the reason behind the constant from the second proof, and one may also figure out where a further improvement may be possible. 

%Our first tool was Lemma \ref{metszet}, which ensures a skew line to $R_i$ having points of large benefit. Observe that we have chosen a skew line to chose from a new point, which seems reasonable since 

To improve further the  bound of Theorem \ref{main} in the main term, one must ensure at least one of the followings.

\begin{itemize}
\item[(1)] Prove the existence of a point that has significantly more benefit then the others, for several steps.
\item[(2)] Prove a lemma, which provides a completion of almost saturating sets to obtain a saturating set with fewer points than Lemma \ref{felezes}. That might enable us to apply Lemma \ref{anal} until a smaller summation limit.
\end{itemize}

Concerning (1), observe that while avoiding lines $\ell$ which are secants of $S_i$ during the choice for the newly added point, it is not guaranteed that the new point won't lie eventually on a secant of $S_i$, so this method does not enable us to provide arcs. However, as these tangents does not contain points for $R_i$, the sum of the benefits on these tangent lines are certainly less than the sum determined in Equation \ref{eq:1}. Despite all this, the known constructions with $O(\sqrt{q})$ points on Galois planes of square order consist of point sets having large secant sizes --- but note that we heavily built on the algebraic (sub)structure.

Concerning (2), notice first that the constant  $\sqrt{3}$ in the bound (and also in the background of the first probabilistic proof) is explained by the application domain of the approaches, as the first phase of both methods lasted until the number of not saturated points decreases from $q^2$ to $O(\sqrt{q})$, so if one cannot exploit the structure of the not saturated points hence cannot obtain a good bound on the variance of the distribution of the benefit values, then Lemma \ref{anal} would provide an evidence that  $O(\sqrt{q})$ points won't form a saturating set in general. 

This problem is strongly connected to the theory of bounding the transversal number of certain hypergraphs.
Recall that for an $r$-uniform hypergraph $\mathcal{H}$, the covering number or \textit{transversal number} $\tau(\mathcal{H})$ is the the minimum cardinality of a set of vertices that intersects all edges of $\mathcal{H}$.

%We are to apply Theorem \ref{qq}.
Consider a point set $S_0$ in the projective plane. Let $X(S_0)=\{x_1, \ldots, x_m\}$ be the set of points not saturated by $S_0$, that is, those points which are not incident to the lines determined by the point pairs of $S_0$.  Assign to each $x_i$ a set $H_i$ of $|S_0|(q-1)+1$ points which make $x_i$ saturated, i.e. $H_i=\left\langle x_i, S_0 \right\rangle\setminus S_0$.\\

Note that the intersection of these sets consists of many points.

\begin{lemma}
	$|H_i\cap H_j|= {|S_0|(|S_0|-1)}$ or  $|H_i\cap H_j|= {(|S_0|-1)(|S_0|-2)}+q$ for every $1\leq i< j\leq m$. 
\end{lemma}

\begin{proof}
	Suppose first that $\left\langle x_i, x_j \right\rangle \cap S_0 = \emptyset$. Then for every $z\neq z' \in S_0$,  the points $\left\langle x_i, z \right\rangle \cap \left\langle x_j, z' \right\rangle$ are distinct, $\left\langle x_i, z \right\rangle \cap \left\langle x_j, z' \right\rangle\in H_i\cap H_j$, and they are points outside $S_0$, hence  $|H_i\cap H_j|= {|S_0|(|S_0|-1)}$.\\
	In the other case  when $|\left\langle x_i, x_j \right\rangle \cap S_0| = 1$, a point $z^* \in S_0$ is incident to $\left\langle x_i, x_j \right\rangle$. Now for every $z\neq z' \in S_0$ with $z\neq z^* \neq z'$,  the points $\left\langle  x_i, z \right\rangle \cap \left\langle x_j, z' \right\rangle\in H_i\cap H_j$ are distinct again and  not belonging to  $S_0$, while  $\left\langle  x_i, z^* \right\rangle= \left\langle x_j, z^* \right\rangle $ thus  $\left\langle  x_i, z^* \right\rangle \setminus z^*$ also belongs to the intersection, hence the claim follows. 
\end{proof}

Lemma \ref{felezes} may thus be altered to a much more general lemma concerning the transversal number of $t$-intersecting uniform hypergraphs, applied to the sets $\{H_i=\left\langle x_i, S_0 \right\rangle\setminus S_0\}$. So an (approximate) solution for the open problem below would imply stronger results for the saturation problem as well.

\begin{problem} Given $\mathcal{F}= \{H_i: i=1\ldots m\}$  an $r$-uniform $t$-intersecting set system on an $n$-element ground set,
prove sharp upper bounds on  $\tau(\mathcal{F})$ in terms of $n, m$, $t$ and $r$.
\end{problem}

It is well known that while the trivial bound $\tau(\mathcal{F})\leq r-t+1$ can be sharp for dense hypergraphs, the approximation of the transversal number is hard in general. Due to Lov\'asz \cite{lovasz}, a connection is made with the fractional transversal  number $\tau^*{(\mathcal{F})}$ (which is also called fractional covering number) as $\tau < (1+\ln{d})\tau^*$, where $d$ is the maximal degree in the hypergraph. 

Note that this result yields instantly  the bound  $\tau < (1+\ln{d})\frac{rm}{r+(m-1)t}$.

By exploiting the intersection property, the greedy algorithm yields that there exists a  point in any $H_i$ that is contained in at least $1+\left\lceil \frac{tm}{r}\right\rceil$ sets from  $\mathcal{F}$. Hence the following proposition follows from the recursion:

\begin{prop}\label{qq} Let $\mathcal{F}= \{H_i: i=1\ldots m\}$ be an $r$-uniform $t$-intersecting  set system on an $n$-element ground set.
Then   $\tau(\mathcal{F})\leq \lceil\frac{rm}{tm+r}\ln{m} \rceil $.
\end{prop}

For further results, we refer to the survey of F\"uredi \cite{furedi2} and the paper of Alon, Kalai and Matou\v{s}ek and Meshulam \cite{Alon1}.
The investigation of the transversal number of $t$-intersecting hypergraphs was proposed also by F\"uredi  in \cite{furedi1}.

 Multiple saturating sets (see e.g. in \cite{Bartoli}) and their generalizations in higher dimensional spaces are also investigated. A point set $S$ in   $\PG(n, q)$ is saturating if any point of $\PG(n, q) \setminus S$ is collinear with
two points in $S$.  
The two proof techniques presented in Sections 2 and 3 are applicable in these more general settings as well.  Namely, the natural analogue of the Lunelli-Sce bound provides the following lower bound for  a saturating set $S$ in  $\PG(n, q)$:

\begin{prop}
$|S|\geq \sqrt{2}q^{\frac{n-1}{2}}$.
\end{prop}

The direct analogue of the first (probabilistic) approach shows the existence of a saturating set in  $\PG(n, q)$ of size $|S|\leq (1+o(1))\sqrt{(n+1)q^{{n-1}}\ln q}$, which improves previously known bounds for $n=4$.

\textbf{Acknowledgement}

Grateful acknowledgement is due to the anonymous referees for their helpful suggestions in order to improve the presentation of the paper.

\end{document}